\documentclass{article}
\usepackage{graphicx,amsfonts,amsmath,mathrsfs,amssymb,amsthm,url,color}
 \usepackage{fancyhdr,indentfirst,bm,enumerate,natbib,dsfont,tikz}
  \usepackage[colorlinks=true,citecolor=blue]{hyperref}
\def\Fbar{ {\overline F}}         
 \def\oo{\infty}                   \def\d{\,\mathrm{d}}
                    
   \def\sg{{\sigma}}                 \def\ep{{\epsilon}}

  \newcommand{\Var}{\mathrm{Var}}
   
    \newcommand{\E}{\mathbb{E}}
     \newcommand{\R}{\mathbb{R}}
      
       \renewcommand{\P}{\mathbb{P}}
       \newcommand{\id}{\mathds{1}}
\renewcommand{\(}{\left (}
 \renewcommand{\)}{\right )}
  \renewcommand{\[}{\left [}
   \renewcommand{\]}{\right ]}
\renewcommand{\(}{\left (}
 \renewcommand{\)}{\right )}
\newtheorem{theorem}{Theorem}[section]
 \newtheorem{corollary}[theorem]{Corollary}
  \newtheorem{lemma}[theorem]{Lemma}
   \newtheorem{proposition}[theorem]{Proposition}
    
     \newtheorem{example}[theorem]{Example}

          \newtheorem{question}{Question}
\numberwithin{equation}{section}
 \numberwithin{theorem}{section}

\renewcommand{\cite}{\citet}
 \allowdisplaybreaks
  
\def\dfrac{\displaystyle \frac}

\begin{document}

\title{Moment inequalities for higher-order (inverse) stochastic dominance }

\author{ Meng Guan\thanks{School of Management, University of Science and Technology of China, China.
Email: \url{gm123123@mail.ustc.edu.cn}}
\and Zhenfeng Zou\thanks{School of Public Affairs, University of Science and Technology of China,  China. Email: \url{zfzou@ustc.edu.cn}}
\and Taizhong Hu\thanks{School of Management, University of Science and Technology of China,  China.
Email: \url{thu@ustc.edu.cn}}          	}

\date{\today}

\maketitle

\begin{abstract}

Stochastic dominance has been studied extensively, particularly in the finance and economics literature. In this paper, we obtain two results. First, necessary conditions for higher-order inverse stochastic dominance are developed. These conditions, which involve moment inequalities of the minimum order statistics, are analogous to the ones obtained by \cite{Fish80b} for usual higher-order stochastic dominance. Second, we investigate how background risk variables influence usual higher-order stochastic dominance. The main result generalizes the ones in \cite{PST20} from the first-order and second-order stochastic dominance to the higher-order.

\medskip

\noindent \textbf{MSC2000 subject classification}: 60E15, 91B06, 91B30


\end{abstract}

\section{Introduction}

Stochastic dominance has been studied extensively in economics, finance, probability, and statistics, among other fields. The literature has been expanding quite rapidly.  We refer the reader to the monographs \cite{MS02}, \cite{SS07}, \cite{Levy16}, and \cite{Wha19} for comprehensive reviews of the stochastic dominance theory.

The \emph{first-order stochastic dominance} (FSD) and \emph{second-order stochastic dominance} (SSD) are the most popular stochastic dominance rules. To more finely characterize a decision-maker's risk behavior or risk preference, research on higher-order stochastic dominance has attracted significant attention. There are two types of higher-order stochastic dominance. The first type, denoted as $n$-SD, is defined based on $n$-th integrated distribution function. For two random variables $X$ and $Y$ with respective distributions $F_X$ and $F_Y$, $X$ is dominated by $Y$ in the sense of $n$-SD, denoted as $X\le_n Y$, if $F_X^{[n]}(x) \ge F_Y^{[n]}(x)$ for all $x\in\R$, where $n\ge 1$ and $F_X^{[n]}$ is the $n$-th integrated distribution function of $F_X$, as defined in \eqref{eq-251228}; see, for example, \cite{Rol76}, \cite{Fish76,Fish80a}, \cite{KHG94}, and \cite{OR01}. The second type, denoted as $n$-ISD, is defined based on $n$-th integrated quantile function. $X$ is dominated by $Y$ in the sense of $n$-ISD, denoted as $X\le_n^{-} Y$, if $F_X^{[-n]}(x) \le F_Y^{[-n]}(x)$ for all $x\in\R$, where $n\ge 1$ and $F_X^{[-n]}$ is the $n$-th integrated quantile function of $F_X$, as defined in \eqref{eq-251229}; see, for example, \cite{MS89}, \cite{WY98}, \cite{MMZ05}, and \cite{DC10}. The formal definitions of $n$-SD and $n$-ISD are given in Section \ref{sect-2} as well as their basic properties.

Exploring the sufficient and/or necessary conditions for higher-order stochastic dominance has always been a major focus of research in this field; see, for example, \cite{JH88a,JH88b,JH88c}, \cite{Fish80b}, \cite{This93}, \cite{CP98}, and \cite{WW25}. A large number of necessary conditions for stochastic dominance have been developed, among which two of the most famous necessary conditions are the moment inequalities typically satisfied by $n$-SD (Theorems \ref{th-181225} and \ref{th-181235}), as given by \cite{Fish80b}.

For $X, Y\in L^1$, if $X<_1 Y$, then $\E [X]< \E [Y]$. For $X, Y\in L^2$, if $X<_2 Y$ with $\E [X]=\E [Y]$, then $\E [X^2] > \E [Y^2]$ and, hence, $\Var(X) > \Var(Y)$. When we consider stochastic dominance between two risk variables $X$ and $Y$, we isolate $X$ and $Y$ from the influence of other unavoidable background risk variables in their environment. These background risks are often uncertainties in economics and finance. \cite{PST20} studied how the FSD and SSD are affected by background noise variables. They proved that (i) if $\E [X] < \E [Y]$ for $X, Y\in L^1$, then there exists a random variable $Z$, independent of $X$ and $Y$, such that $X+Z <_1 Y+Z$, and (ii) if $\E [X] =\E [Y]$ and $\E [X^2]> \E [Y^2]$ for $X, Y\in L^2$, then there exists a random variable $Z$, independent of $X$ and $Y$, such that $X+Z <_2 Y+Z$.

In this paper, we will further investigate the properties of $n$-SD and $n$-ISD for $n>2$. The main contributions of this paper are as follows:

\begin{itemize}
  \item[1.] Motivated by Theorems \ref{th-181225} and \ref{th-181235}, we establish moment inequalities of the minimum order statistics for $n$-ISD when $n>2$ (Theorems \ref{th-251226} and \ref{th-251227}).

  \item[2.] Parallel to the results mentioned above in \cite{PST20}, we investigate how background risk variables influence $n$-SD for $n>2$. The main result (Theorem \ref{th-251212}) demonstrates that background risk can be strong enough to reinforce the ordering of two variables in terms of $n$-SD, even when they are ranked based on their moments for $n> 2$.

  \item[3.] Alternative proofs of the moment inequalities (Theorems \ref{th-181225} and \ref{th-181235}) for $n$-SD are presented by placing greater emphasis on the idea of the upper and lower asymptotes of integrated distribution functions.
\end{itemize}

\noindent We believe the main results in this paper contribute to a better understanding of the meaning of the $n$-SD and $n$-ISD.

The paper is organized as follows. In Section \ref{sect-2}, we review the fundamental concepts of usual $n$-SD and $n$-ISD, where $n>2$, and present auxiliary results for these two types of stochastic dominance, aiming to derive the moment inequalities of the minimum order statistics for $n$-ISD in Section \ref{sect-3}. The usual $n$-SD under independent noise variable for $n>2$ is investigated in Section \ref{sect-4}. In Appendix \ref{sect-proofs}, we present alternative proofs of the moment inequalities for $n$-SD.

\section{Preliminaries}
\label{sect-2}

\subsection{The $n$-SD}

For $r\ge 0$, let $L^r$ be the set of all random variables on an atomless probability space $(\Omega, \mathscr{F}, \P)$ with finite $r$th moment. For $X\in L^0$ with distribution function $F_X$, set $F_X^{[1]}(x)=F_X(x)$, and define $n$-th integrated distribution function $F^{[n]}_X$ recursively by
\begin{equation}
 \label{eq-251228}
     F^{[n]}_X (x) = \int_{-\infty}^x F_X^{[n-1]} (t) \d t,\quad x\in \R,\ n\ge 2,
\end{equation}
which is also termed as higher-order cumulative function. It is well-known that \citep[see, for example,][Proposition 1]{OR01}
\begin{equation}
  \label{eq-181228-1}
    F_X^{[n]}(x)=\frac {1}{(n-1)!} \int^x_{-\oo} (x-y)^{n-1} \d F_X(y)= \frac {1}{(n-1)!}\, \E [(x-X)_+^{n-1}],\quad  x\in\R.
\end{equation}

Similar to the $n$-th integrated distribution function, we can define $n$-th integrated survival function $\widetilde{F}_X^{[n]}(x)$ as follows. Denote by $\widetilde{F}_X^{[1]}(x)=S_X(x)=1-F_X(x)$, the survival function of $X$, and define $n$-th integrated survival function $\widetilde{F}^{[n]}_X$ recursively by
$$
     \widetilde{F}^{[n]}_X (x) = \int_x^{\infty} \widetilde{F}_X^{[n-1]} (t) \d t,\quad x\in \R\ \hbox{and}\ n\ge 2.
$$
For $X\in L^{n-1}$, $n\ge 1$, a similar argument to the proof of Proposition 1 in \cite{OR01} yields that
\begin{equation}
 \label{eq-181228-2}
 \widetilde{F}_X^{[n]}(x)=\frac {1}{(n-1)!} \int_x^\oo (y-x)^{n-1} \d F_X(y)
    = \frac {1}{(n-1)!}\, \E [(X-x)_+^{n-1}],\quad  x\in\R.
\end{equation}
From \eqref{eq-181228-1} and \eqref{eq-181228-2}, $F_X^{[n]}(x)$ and $\widetilde{F}_X^{[n]}(x)$ are finite for any $x\in\R$ when $X\in L^{n-1}$ and $n\ge 1$.

We recall from \cite{Fish76,Fish80a} the notion of \emph{$n$-th degree stochastic dominance} ($n$-SD). For $X, Y\in L^{n-1}$ and $n\ge 1$, $X$ is said to be dominated by $Y$ in the sense of $n$-SD, denoted by $X\le_n Y$, if
\begin{equation}
  \label{eq-1201-1}
    F_X^{[n]}(x) \ge F_Y^{[n]}(x),\quad x\in\R.
\end{equation}
If $X\le_n Y$ and the strict inequality in \eqref{eq-1201-1} holds for at least one point $t_0\in\R$, we say $X$ is strictly dominated by $Y$, denoted by $X<_n Y$. That is, $X<_n Y$ if and only if $X\le_n Y$ but $Y\not\le_n X$. For $n=1$ and $2$, $n$-SD reduces to FSD and SSD, respectively. Clearly, $X\le_n Y$ implies $X\le_{n+1} Y$, and $X<_n Y$ implies $X<_{n+1} Y$.

The $n$-SD defined by \eqref{eq-1201-1} is just one formulation of higher-order stochastic dominance with reference interval $\R$. There is another formulation of higher-order stochastic dominance, which was initially introduced by \cite{Jean80} and has been widely adopted in decision theory \citep[see, for example,][]{EST09}. This formulation is applied to all distributions with bounded interval $[a, b]$.
This criterion requires that $F_X^{[n]}(x) \ge F_Y^{[n]}(x)$ for all $x\in [a,b]$, and $F_X^{[k]}(b) \ge F_Y^{[k]}(b)$ for $k=1, \ldots, n-1$. Regarding the differences between these two types of higher-order stochastic dominance and their characterizations based on expected utility for some sets of utility functions, please refer to \cite{WW25}. In this paper, we focus on the first formulation of higher-order stochastic dominance with reference interval $\R$.

\cite{Fish80b} proved Theorems \ref{th-181225} and \ref{th-181235} under the assumption that $X$ and $Y$ are nonnegative random variables. \cite{Obr84} proved Theorems \ref{th-181225} and \ref{th-181235} for general random variables by establishing a relationship between the behavior of $G^{[n]}(x)- F^{[n]}(x)$ for large $x$ and certain inequalities involving the moments of $F$ and $G$. We will present the proofs of these two theorems in Appendix \ref{sect-proofs}. Although our proof method is similar to that of \cite{Obr84}, we place greater emphasis on the idea of the upper and lower asymptotes of $F^{[n]}_X(x)$ and $F_Y^{[n]}(x)$.

\begin{theorem}
\label{th-181225}
{\rm \citep{Fish80b,Obr84}}.
Let $X$ and $Y$ be two random variables (unnecessarily nonnegative) such that $X\ge_n Y$ and
$X, Y\in L^{n-1}$ for $n\ge 2$. If $\E [X^j]=\E [Y^j]$ for $j=0,1,\ldots, k$, $0\le k<n-1$, then
\begin{equation}
  \label{eq-181228-8}
      (-1)^k\, \E [X^{k+1}] \ge (-1)^k\, \E [Y^{k+1}].
\end{equation}
\end{theorem}

\begin{theorem}
\label{th-181235}
{\rm \citep{Fish80b,Obr84}}.
Let $X$ and $Y$ be two random variables (unnecessarily nonnegative) such that $X>_n Y$ and $X, Y\in L^n$, where $n\ge 1$. If $\E [X^k]=\E [Y^k]$ for $k=0,1,\ldots, n-1$, then
\begin{equation}
  \label{eq-181229-4}
       (-1)^{n-1} \E [X^n] > (-1)^{n-1}\E [Y^n].
\end{equation}
\end{theorem}

From Theorem \ref{th-181225}, it follows that $X\ge_n Y$ always implies $\E [X]\ge \E [Y]$ for $n\ge 1$. A short proof for $n=3$ was given by \cite{Schm05}. Theorem \ref{th-181235} implies that $(\E [X],\ldots, \E [X^n]) = (\E [Y],\ldots, \E [Y^n])$ can not be true when $X>_n Y$, provided that the moments involved are finite. Since $X \ge_m Y$ implies $X \ge_{m+1} Y$ for $m\ge 1$, an immediate consequence of Theorem \ref{th-181225} is the following corollary, which states that if $X$ stochastically dominates $Y$ in the non-strict sense for any finite degree and if $\E [X^k]=\E [Y^k]$ for $k=0,1,\ldots, n$, then $\E [X^{n+1}] \ge \E [Y^{n+1}]$ for even $n$ and $\E [X^{n+1}] \le \E [Y^{n+1}]$ for odd $n$.

\begin{corollary}
\label{co-181225}
Let $X$ and $Y$ be two random variables (unnecessarily nonnegative) such that $X\ge_m Y$ for $m\ge 1$. If $X, Y\in L^{n+1}$, $n\ge 0$, and $\E [X^k]=\E [Y^k]$ for $k=0,1,\ldots, n$, then
$$     (-1)^n\, \E [X^{n+1}] \ge (-1)^n\, \E [Y^{n+1}].    $$
\end{corollary}

\subsection{The $n$-ISD}

For $X\in L^0$ with distribution function $F_X$, the left-continuous inverse of $F_X$ is defined by
$$
   F_X^{-1}(p)=\inf\{x\in\R: F_X(x)\ge p\},\quad p\in (0,1],
$$
with $F_X^{-1}(0)=\inf\{x\in\R: F_X(x)>0\}$. Throughout, we work within the space $L^1$ of random variables. For $X\in L^1$, starting from $F_X^{[-1]}(p) = F_X^{-1}(p)$ we define $n$-th quantile function $F_X^{[-n]}$ recursively from
\begin{equation}
 \label{eq-251229}
   F_X^{[-n]}(p) =\int^p_0 F^{[-n+1]}_X (u) \d u, \quad p\in [0,1],\ n\ge 2.
\end{equation}

First, we recall from \cite{MS89} the \emph{$n$-th degree inverse stochastic dominance} ($n$-ISD). A random variable $X$ is said to be dominated by $Y$ in the sense of $n$-ISD, denoted by $X\le^{-}_n Y$, if
\begin{equation}
 \label{eq-251227}
    F_X^{[-n]} (p) \le F_Y^{[-n]} (p),\quad p\in (0,1),\ n\ge 1.
\end{equation}
If $X\le_n^{-} Y$ and the strict inequality in \eqref{eq-251227} holds for at least one point $p_0\in (0,1)$, we say $X$ is strictly dominated by $Y$ in the sense of $n$-ISD, denoted by $X<_n^{-} Y$.
It is known that $n$-ISD and $n$-SD are equivalent for $n=1, 2$. When $n\ge 3$, a counterexample was given by \cite{SF87} to show this equivalence does not hold. The difference between $n=1, 2$ and $n\ge 3$ is precisely one of the key points discussed in \cite{DC10}, which also provides an equivalent characterization of the $n$-ISD order in terms of the weak $n$-majorization of the vectors of mean order statistics. Here, the $n$-ISD is also termed in \cite{AHM21} as the \emph{$n$-th degree upward inverse stochastic dominance} (upward $n$-ISD).

Since $F_X^{[-n]}(p)$ and $F_Y^{[-n]}(p)$ recursively integrate the quantile functions over the interval $(0, p]$, the $n$-ISD compares the lowest part of the distributions. This is the reason why $n$-ISD has gained attention in the study of social inequality and social welfare, as the focus is on the strata within these research fields that have access to fewer resources. See, for example, \cite{Zoli99,Zoli02}, \cite{And18}, \cite{AHM21}, and \cite{JSH24} among references therein.

\begin{proposition}
 \label{pr-251220}
For $X \in L^1$ and $n\ge 2$, we have
\begin{equation}
  \label{eq-251220-1}
   F_X^{[-n]} (p) = \frac {1}{(n-2)!} \int^1_0 F_X^{-1}(u) (p-u)_+^{n-2}\d u, \quad p\in [0,1].
\end{equation}
If, additionally, $F_X$ is continuous, then
\begin{equation}
 \label{eq-251220-2}
    F_X^{[-n]} (p)  =\frac {1}{(n-2)!} \int_{\R} x \big (p-F_X(x)\big )_+^{n-2}\d F_X(x), \quad p\in [0,1].
\end{equation}
\end{proposition}

\begin{proof}
We prove the proposition by induction on $n$. Eq. \eqref{eq-251220-1} also appears in \cite{And18}. For completeness, we give its proof. For $n=2$, \eqref{eq-251220-1} is trivial. Assuming that it holds for $n=k\ge 2$, we shall show it for $k+1$. We have
\begin{align*}
    F_X^{[-k-1]} (p) & = \frac {1}{(k-2)!} \int^p_0 \(\int^v_0 F_X^{-1}(u) (v-u)^{k-2}\d u\) \d v \\[3pt]
    &= \frac {1}{(k-2)!} \int^p_0 \(\int^p_u F_X^{-1}(u) (v-u)^{k-2}\d v\) \d u \\[3pt]
    & = \frac {1}{(k-1)!} \int^p_0 F_X^{-1}(u) (p-u)^{k-1}\d u,
\end{align*}
where the order of integration could be changed by Fubini's theorem. Thus, \eqref{eq-251220-1} holds by induction.

Next, we turn to prove \eqref{eq-251220-2}. First observe that there exists a standard uniform random variable $U_X$ defined on the space $(\Omega, \mathscr{F}, \P)$ such that $X=F_X^{-1}(U_X)$ almost surely. The existence of $U_X$ was given in Lemma A.28 of \cite{FS11}; also see \cite{Delb12}. Since $F_X$ is continuous, it follows that $F_X^{-1}(u)$ is strictly increasing in $u\in (0,1)$. Then, For $n=2$, we have
\begin{align*}
  F_X^{[-2]} (p) & =\E \[F_X^{-1}(U_X) \id_{\{U_X <p\}} \]
     =\E \[F_X^{-1}(U_X) \id_{\{F^{-1}_X(U_X) < F_X^{-1}(p)\}} \] \\
    & =\E \[X \id_{\{X < F_X^{-1}(p)\}} \]  =\E \[X \id_{\{F_X(X) < p\}} \],
\end{align*}
implying \eqref{eq-251220-2} with $n=2$. Assuming that \eqref{eq-251220-2} holds for $n=k\ge 2$, we shall show it for $k+1$. We have
\begin{align*}
   F_X^{[-k-1]} (p) & = \frac {1}{(k-2)!} \int^p_0 \(\int_{\R} x \big (u-F_X(x)\big )_+^{k-2}\d F_X(x)\)\d u\\
   & = \frac {1}{(k-2)!} \int_{\R} x \(\int^p_0 \big (u-F_X(x)\big )_+^{k-2} \d u \) \d F_X(x)\\
   & =\frac {1}{(k-1)!} \int_{\R} x \big (p-F_X(x)\big )_+^{k-1}\d F_X(x),
\end{align*}
where the order of integration could be changed by Fubini's theorem. Thus, \eqref{eq-251220-2} follows by induction.
\end{proof}

From Proposition \ref{pr-251220}, it is seen that $ F_X^{[-n]} (p)$ is well defined, and $|F_X^{[-n]} (p)|\le \E |X|/(n-2)!$ for $n\ge 2$ and $p\in [0,1]$ when $X\in L^1$. Compared with $n$-SD, this is one of the advantages of $n$-ISD, because when considering $n$-SD, we need to assume that the random variables have finite $(n-1)$-th moments.

If $F_X$ is not continuous, \eqref{eq-251220-2} may be not true as illustrated by the following counterexample. Let $F_X$ have a jump at point $x_0\ne 0$ such that $p_0= F_X(x_0-) < F_X(x_0)=p_1$. Then, for $p\in (p_0, p_1)$,
$$
   F^{[-2]}_X(p)=\int^p_0 F_X^{-1}(u)\d u =\int^{p_0}_0 F_X^{-1}(u)\d u + x_0 (p-p_0),
$$
and the right hand side of \eqref{eq-251220-2} is
\begin{align*}
    \int_{\R} x \id_{\{F_X(x)<p\}} \d F_X(x)  &=\int_{\R} x \id_{\{x < x_0\}} \d F_X(x)
     =\E \[X \id_{\{X< x_0\}}\] \\
     & =\E \[F_X^{-1}(U_X) \id_{\{F_X^{-1}(U_X)< x_0\}}\] \\
     & = \E \[F_X^{-1}(U_X) \id_{\{U_X< p_0\}}\] \\
     & =\int^{p_0}_0 F_X^{-1}(u)\d u \ne  F^{[-2]}_X(p),
\end{align*}
which implies \eqref{eq-251220-2} is not true in this case.

Analogously, denote $\widetilde{F}_X^{[-1]}(p) = F_X^{-1}(p)$, and define recursively
$$
   \widetilde{F}_X^{[-k]}(p) = \int_p^1 \widetilde{F}_X^{[-k+1]}(u) \d u,\quad p \in [0,1],\ k\ge 2.
$$
By comparing $\widetilde{F}_X^{[-k]}(p)$ and $\widetilde{F}_Y^{[-k]}(p)$ for all $p\in [0,1]$, we can define another ordering, which was respectively called the \emph{dual stochastic dominance} in \cite{WY98}, and the \emph{$n$-th degree downward inverse stochastic dominance} (downward $n$-ISD) in \cite{AHM21}. From now on, we consider the $n$-ISD ordering defined by \eqref{eq-251227}, as the results of $n$-ISD can be corresponding transformed into the downward $n$-ISD case.

Similarly, we have the following proposition, where \eqref{eq-251226-6} is (4.3) in \cite{ZH24}.  Eqs. \eqref{eq-251226-6} and \eqref{eq-251220-1} will be used in the proof of Theorem \ref{th-251226}.

\begin{proposition}
 \label{pr-251221}
For $X \in L^1$ and $n\ge 2$, we have
\begin{align}
 \label{eq-251226-6}
  \widetilde{F}_X^{[-n]} (p) & = \frac {1}{(n-2)!} \int_0^1 F_X^{-1}(u) (u-p)_+^{n-2}\d u,\quad p\in [0,1].
\end{align}
If, additionally, $F_X$ is continuous, then
\begin{align*}
  \widetilde{F}_X^{[-n]} (p) & =\frac {1}{(n-2)!} \int_{\R} x\big (F_X(x)-p\big )_+^{n-2}\d F_X(x), \quad p\in [0,1].
\end{align*}
\end{proposition}

The next proposition gives alternative integral representations of $F_X^{[-n]}(p)$ and $\widetilde{F}_X^{[-n]}(p)$. For a nonnegative random variable $X\in L^1$, \eqref{eq-251225-2} is given in Lemma 4.5 of \cite{WY98}.

\begin{proposition}
 \label{pr-251226}
For $X\in L^1$, $p\in (0,1)$ and $n\ge 2$, we have
\begin{align}
  F_X^{[-n]}(p) & =\frac {1}{(n-1)!} \[ \int^\oo_0 (p-F_X(t))_+^{n-1} \d t -\int^0_{-\oo}
           \Big ( p^{n-1}-(p-F_X(t))_+^{n-1}\Big ) \d t \],         \label{eq-251225-1} \\[3pt]
  \widetilde{F}_X^{[-n]}(p) & =\frac {1}{(n-1)!} \[ \int^\oo_0 \Big ((1-p)^{n-1} -(F_X(t)-p)_+^{n-1}\Big ) \d t            -\int^0_{-\oo} (F_X(t)-p)_+^{n-1}\d t \].    \label{eq-251225-2}
\end{align}
\end{proposition}

\begin{proof}
We prove the proposition by induction on $n$. For $n=2$,
\begin{align*}
  F_X^{[-2]}(p) & = \int^\oo_0 (p-F_X(t))_+ \d t -\int^0_{-\oo} \Big ( p-(p-F_X(t))_+\Big ) \d t,\\[3pt]
  \widetilde{F}_X^{[-2]}(p) & =\int^\oo_0 \Big (1-p -(F_X(t)-p)_+ \Big ) \d t
            -\int^0_{-\oo} (F_X(t)-p)_+ \d t,
\end{align*}
as can be seen from Figure \ref{Fig-1}.
\begin{figure}[htbp]
  \centering
  \includegraphics[scale=0.25]{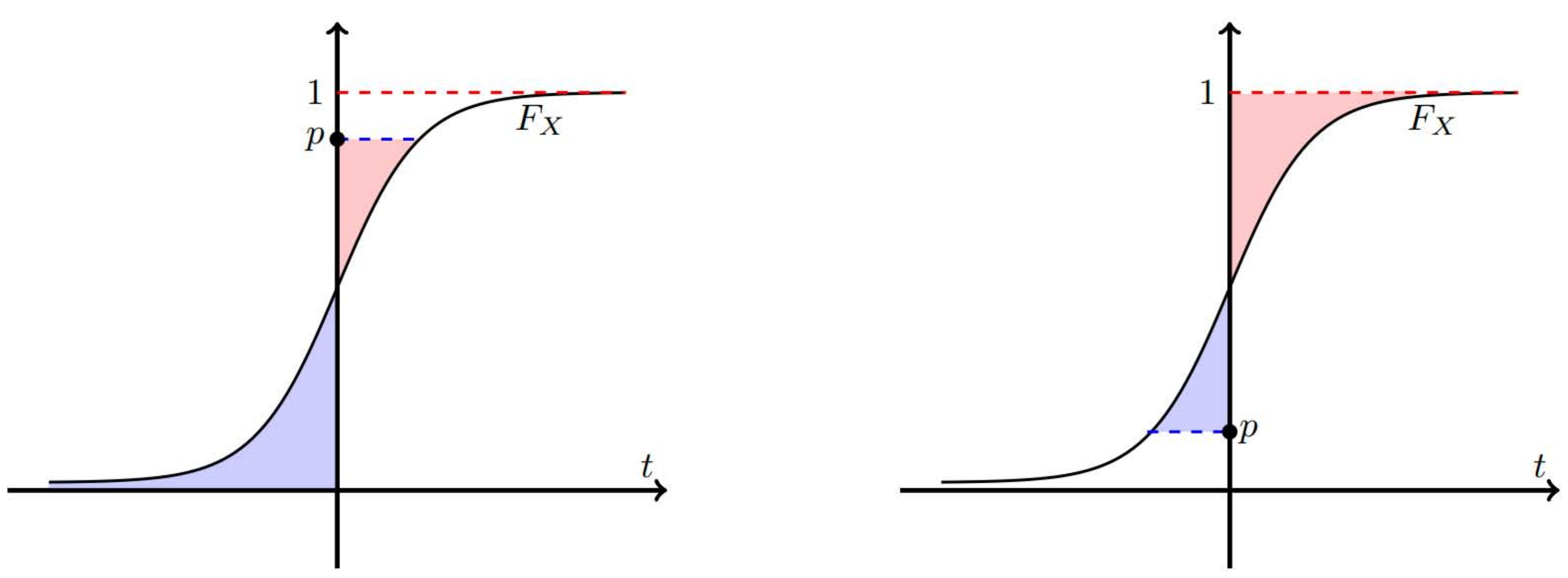}
  \caption{Representations of $F_X^{[-2]}(p)$ and $\widetilde{F}_X^{[-2]}(p)$}
   \label{Fig-1}
\end{figure}
This means that \eqref{eq-251225-1} and \eqref{eq-251225-2} are true for $n=2$. Assume that
\eqref{eq-251225-1} and \eqref{eq-251225-2} hold for $n=k\ge 2$. For $n=k+1$, we have
\begin{align*}
  k! F_X^{[-k-1]}(p) & =\int^p_0 k! F_X^{[-k]}(u) \d u  \\[3pt]
  & = k \int^\oo_0 \[ \int^p_0 (u\!-\! F_X(t))_+^{k-1}\d u \] \d t - k\int^0_{-\oo} \[\int^p_0
           \Big ( u^{k-1}-(u\! -\! F_X(t))_+^{k-1}\Big ) \d u\] \d t  \\[3pt]
  & = \int^\oo_0 (p-F_X(t))_+^k \d t - k\int^0_{-\oo} \Big ( p^k-(p-F_X(t))_+^k\Big ) \d t,
\end{align*}
and
\begin{align*}
 k! \widetilde{F}_X^{[-k-1]}(p)  & =\int_p^1 k! \widetilde{F}_X^{[-k]}(u) \d u  \\[3pt]
  & = k\int^\oo_0 \[\int^1_p\Big ( (1\!-\! u)^{k-1} -(F_X(t)\!-\! u)_+^{k-1}\Big )\d u\] \d t
          - k\int^0_{-\oo}\[\int^1_p (F_X(t)\! -\! u)_+^{k-1}\d u\]\d t \\[3pt]
  & = \int^\oo_0 \Big ( (1-p)^k -(F_X(t)- p)_+^k\Big ) \d t - \int^0_{-\oo} (F_X(t) - p)_+^k\d t.
\end{align*}
Therefore, \eqref{eq-251225-1} and \eqref{eq-251225-2} hold for $n=k+1$. This proves the proposition by induction.
\end{proof}

\section{Moment inequalities for $n$-ISD}
\label{sect-3}

It is well known that $\E [X]\le \E [Y]$ whenever $X\le^{-}_2 Y$. The following two examples demonstrate that the relation $X\le^{-}_3 Y$ may hold whether $\E [X]>\E[Y]$ or $\E[X] <\E[Y]$.

\begin{example}{\rm
($X\le^{-}_3 Y$ and $\E [X] >\E [Y]$)\ \ Consider two random variables $X$ and $Y$ with probability mass functions $\P(X=0) =0.5 =1-\P(X=10)$ and $\P(Y=4)=0.9 =1-\P(Y =4.1)$. Then $\E[X]=5>4.01 =\E [Y]$, and
\begin{equation*}
  F_X^{-1}(p) =\left\{\begin{array}{ll} 0, & p \in [0, 0.5], \\ 10, & p \in (0.5, 1], \end{array} \right. \qquad
    F_Y^{-1}(p)=\left\{\begin{array}{ll} 4, & p\in [0, 0.9], \\ 4.1, & p \in (0.9, 1]. \end{array} \right.
\end{equation*}
By Proposition \ref{pr-251220}, we have
\begin{equation*}
   F_X^{[-3]}(p)= \int_0^p F_X^{-1}(u)(p-u)\d u = \left\{\begin{array}{ll} 0, & p\in [0, 0.5], \\
    5 (p - 0.5)^2, & p\in (0.5, 1], \end{array} \right.
\end{equation*}
and
\begin{equation*}
  F_Y^{[-3]}(p)= \int_0^p F_Y^{-1}(u) (p-u)\d u= \left\{\begin{array}{ll} 2 p^2, & p\in [0, 0.9], \\
   (p-0.9)^2/20 + 2 p^2, & p \in (0.9, 1], \end{array} \right.
\end{equation*}
It is straightforward to verify that $F_X^{[-3]}(p)\le F_Y^{[-3]}(p)$ for all $p\in [0,1]$. Thus, $X\le^{-}_3 Y$.  }
\end{example}

Examples 1 and 2 in \cite{DC10} also satisfy $X\le_3^{-1} Y$ and $\E [X] >\E [Y]$.

\begin{example}{\rm
($X\le^{-}_3 Y$ and $\E [X] <\E [Y]$)\ \ Consider two random variables $X$ and $Y$ with $\P(X=1) =0.5 =1-\P (X = 3)$ and $\P(Y=2.5)=1$. Then $\E [X]=2<2.5=\E[Y]$, and
\begin{equation*}
   F_X^{-1}(p) = \left\{\begin{array}{ll} 1, & p \in [0, 0.5], \\ 3, & p\in (0.5,1], \end{array} \right. \qquad
   F_Y^{-1}(p) = 2.5.
\end{equation*}
Thus,
\begin{equation*}
   F_X^{[-3]}(p) =\int_0^p F_X^{-1}(u) (p-u)\d u = \left\{\begin{array}{ll} 0.5 p^2, & p\in [0, 0.5], \\
     1.5 p^2 -p+ 0.25, & p\in (0.5, 1], \end{array} \right.
\end{equation*}
and
\begin{equation*}
     F_Y^{[-3]}(p) = \int_0^p F_Y^{-1}(u) (p-u) \d u = 1.25 p^2,\quad p\in [0,1].
\end{equation*}
It is straightforward to verify that $F_X^{[-3]}(p)\le F_Y^{[-3]}(p)$ for all $p\in [0,1]$. Thus, $X\le^{-}_3 Y$. }
\end{example}

Let $\{X_n, n\ge 1\}$ be a sequence of independent and identically distributed random variables with a common distribution $F_X$, and set $X_{1:k}=\min\{X_1, \ldots, X_k\}$ for $k\ge 1$. For convenience, denote $\mu^X_{1:k} =\E [X_{1:k}]$. Since $F_{X_{1:k}}(x)=1- [S_X(x)]^k$, it follows that
$F_{X_{1:k}}^{-1}(p)= F_X^{-1} (1-(1-p)^{1/k})$ for $p\in (0,1)$. Thus,
\begin{equation}
 \label{eq-251226-7}
   \mu^X_{1:k} =\int^1_0 F_X^{-1} (1-(1-p)^{1/k}) \d p =k\int^1_0 F_X^{-1}(u) (1-u)^{k-1}\d u =  k! F_X^{[-k-1]}(1).
\end{equation}
Another two expressions of $ \mu^X_{1:k}$ are as follows:
\begin{align*}
    \mu^X_{1:k} & =k\int_{\R} x [S_X(x)]^{k-1} \d F_X(x) \\[3pt]
    & = \int^\oo_0 [S_X(x)]^k \d x - \int^0_{-\oo} \big (1-[S_X(x)]^k\big ) \d x.
\end{align*}

The next proposition provides necessary conditions for $n$-ISD, whose special case for nonnegative random variables is Theorem 1 in \cite{MS89}.

\begin{proposition}
If $X\le^{-}_n Y$ for $n>2$, where $X, Y\in L^1$, then $\mu^X_{1:k} \le \mu^Y_{1:k}$ for all $k \ge n-1$.
\end{proposition}

\begin{proof}
Since $X\le^-_n Y$ implies $X\le^-_{n+1} Y$, it sufficed to prove $\mu^X_{1:(n-1)} \le \mu^Y_{1:(n-1)}$ when $X\le^-_n Y$. The desired result now follows from \eqref{eq-251226-7} and the definition of $n$-ISD.
\end{proof}

\begin{theorem}
 \label{th-251226}
Let $X, Y\in L^1$ satisfying $X\le_n^- Y$ for $n>2$. If
\begin{equation}
 \label{eq-251226-1}
    \mu^X_{1:(n-1-j)}=\mu^Y_{1:(n-1-j)}\ \hbox{for}\ j=0, \ldots, k,\ 0\le k<n-2.
\end{equation}
then
\begin{equation}
 \label{eq-251226-2}
   (-1)^{k+1}\, \mu^X_{1:(n-2-k)} \le (-1)^{k+1} \,\mu^Y_{1:(n-2-k)}.
\end{equation}
\end{theorem}

\begin{proof}
We will prove the desired result \eqref{eq-251226-2} by contradiction. Assume on the contrary that \eqref{eq-251226-2} does not hold, that is,
\begin{equation}
 \label{eq-251226-3}
   (-1)^{k+1}\, \mu^X_{1:(n-2-k)} > (-1)^{k+1} \,\mu^Y_{1:(n-2-k)}.
\end{equation}
We consider two cases based on whether $n$ is odd or even.

First, assume $n=2m+1>2$ is odd. By \eqref{eq-251220-1} and \eqref{eq-251226-6}, we have
\begin{align}
  A^X_n(p) & :=  F_X^{[-n]} (p) - \widetilde{F}_X^{[-n]} (p) \nonumber \\
  & = \frac {1}{(n-2)!} \int^1_0 F_X^{-1}(u) \[(p-u)_+^{n-2} - (u-p)_+^{n-2}\] \d u \nonumber  \\[3pt]
  & = \frac {1}{(n-2)!} \int^1_0 F_X^{-1}(u) \[\((p-u)^{n-2}\)_+ - \((p-u)^{n-2}\)_-\] \d u \nonumber \\[3pt]
  & = \frac {1}{(n-2)!} \int^1_0 F_X^{-1}(u) (p-u)^{n-2} \d u \nonumber \\[3pt]
  &=\frac {(1-p)^{n-2}}{(n-2)!} \sum^{n-2}_{j=0} (-1)^{j+1} {n-2\choose j} (1-p)^{-j} \int^1_0 F_X^{-1}(u) (1-u)^j\d u  \nonumber \\[3pt]
  &=\frac {(1-p)^{n-2}}{(n-1)!} \sum^{n-2}_{j=0} (-1)^{j+1} {n-1\choose j+1} (1-p)^{-j} \int^1_0 (j+1) F_X^{-1}(u) (1-u)^j\d u  \nonumber \\[3pt]
  &=\frac {(1-p)^{n-2}}{(n-1)!} \sum^{n-2}_{j=0} (-1)^{j+1} {n-1\choose j+1} (1-p)^{-j} \mu^X_{1:(j+1)}  \nonumber \\[3pt]
  & = \frac {(1-p)^{n-1}}{(n-1)!} \sum^{n-1}_{j=1} (-1)^j {n-1\choose j} (1-p)^{-j} \mu^X_{1:j}.
    \label{eq-251230-2}
\end{align}


\noindent Therefore, by \eqref{eq-251226-1} and \eqref{eq-251226-3},
\begin{align*}
   & \frac {(n-1)!}{(1-p)^{n-2}} \(A^X_n(p) -A^Y_n(p)\) \\
   & \quad = (1-p) \sum^{n-2-k}_{j=1} (-1)^j {n-1\choose j} (1-p)^{-j} \big (\mu^X_{1:j}-\mu^Y_{1:j}\big ) \\
   & \quad = {n-1\choose n-2-k} (-1)^{k+1} \(\mu^X_{1:(n-2-k)}-\mu^Y_{1:(n-2-k)}\)  (1-p)^{-(n-3-k)} [1 + o (1)]  \\[3pt]
   & \quad \longrightarrow \delta_k,\quad p\to 1,
\end{align*}
where $\delta_k = {n-1\choose n-2-k} (-1)^{k+1} \big (\mu^X_{1:(n-2-k)}-\mu^Y_{1:(n-2-k)}\big )>0$ for $k=n-3$, and $\delta_k=+\oo$ for $k<n-3$. On the other hand, as $p\to 1$,
$$
  \frac{\left |\widetilde{F}_X^{[-n]}(p)\right |}{(1-p)^{n-2}}
     \le \frac {1}{(n-2)!} \int_p^1 \left | F_X^{-1}(u)\right | \(\frac {u-p}{1-p}\)^{n-2}\d u
       \le \frac {1}{(n-2)!} \int_p^1 \left | F_X^{-1}(u)\right |\d u \to 0.
$$
Similarly, $ \left |\widetilde{F}_Y^{[-n]}(p)\right |/(1-p)^{n-2}\to 0$ as $p\to 1$. So, we get that
\begin{equation}
 \label{eq-251222-4}
    \frac {(n-1)!}{(1-p)^{n-2}} \(F_X^{[-n]}(p) -F_Y^{[-n]}(p)\) \longrightarrow \delta_k>0,\quad p\to 1,
\end{equation}
implying that there exists $p _0\in (0,1)$ such that $F_X^{[-n]}(p) > F_Y^{[-n]}(p)$ for $p>p_0$. This is a contradiction.

Next, assume $n=2m>2$. Then
\begin{align}
  B_n^X(p) & :=  F_X^{[-n]} (p) + \widetilde{F}_X^{[-n]} (p)  \nonumber  \\
   & = \frac {1}{(n-2)!} \int^1_0 F_X^{-1}(u) (p-u)^{n-2} \d u  \nonumber \\[3pt]
  &=\frac {(1-p)^{n-2}}{(n-2)!} \sum^{n-2}_{j=0} (-1)^j {n-2\choose j} (1-p)^{-j} \int^1_0 F_X^{-1}(u) (1-u)^j\d u  \nonumber \\[3pt]
  &=\frac {(1-p)^{n-2}}{(n-1)!} \sum^{n-2}_{j=0} (-1)^j {n-1\choose j+1} (1-p)^{-j} \int^1_0 (j+1) F_X^{-1}(u) (1-u)^j\d u  \nonumber \\[3pt]
  &=\frac {(1-p)^{n-2}}{(n-1)!} \sum^{n-2}_{j=0} (-1)^j {n-1\choose j+1} (1-p)^{-j} \mu^X_{1:(j+1)} \nonumber \\[3pt]
  & = - \frac {(1-p)^{n-1}}{(n-1)!} \sum^{n-1}_{j=1} (-1)^j {n-1\choose j} (1-p)^{-j} \mu^X_{1:j}.  \label{eq-251230-3}
\end{align}
Again,  by \eqref{eq-251226-1} and \eqref{eq-251226-3}, we have
\begin{align*}
   & \frac {(n-1)!}{(1-p)^{n-2}} \(B_n^X(p) - B_n^Y(p)\) \\
   & \qquad = -(1-p) \sum^{n-2-k}_{j=1} (-1)^j {n-1\choose j} (1-p)^{-j} \big (\mu^X_{1:j}-\mu^Y_{1:j}\big ) \\
   & \qquad =- {n-1\choose n-2-k} (-1)^{n-2-k} \(\mu^X_{1:(n-2-k)}-\mu^Y_{1:(n-2-k)}\)  (1-p)^{-(n-3-k)} [1 + o (1)]  \\[3pt]
   & \qquad = {n-1\choose n-2-k} (-1)^{k+1} \(\mu^X_{1:(n-2-k)}-\mu^Y_{1:(n-2-k)}\)  (1-p)^{-(n-3-k)} [1 + o (1)]  \\[3pt]
   & \qquad \longrightarrow \delta_k,\quad p\to 1.
\end{align*}
Similarly, we conclude \eqref{eq-251222-4}, which is a contradiction. Therefore, we complete the proof of the theorem.
\end{proof}
	
Theorem \ref{th-251226} can be proved by using \eqref{eq-251225-1} and \eqref{eq-251225-2}, instead of \eqref{eq-251220-1} and \eqref{eq-251226-6}. According to Theorem \ref{th-251226},  if $X\le^{-}_n Y$ and $\mu^X_{1:(n-1)}=\mu^Y_{1:(n-1)}$, then $\mu^X_{1:(n-2)}\ge \mu^Y_{1:(n-2)}$. However, the following two examples show that when the condition $\mu^X_{1:(n-1)}=\mu^Y_{1:(n-1)}$ is removed, the relationship between $\mu^X_{1:(n-2)}$ and $\mu^Y_{1:(n-2)}$ is uncertain --- it can be either $\mu^X_{1:(n-2)}\ge  \mu^Y_{1:(n-2)}$ or $\mu^X_{1:(n-2)}< \mu^Y_{1:(n-2)}$.

\begin{example}
{\rm Consider a random variables $X$ satisfying $\P(X =\pm a) = 0.5$ with $a>0$, and let $Y=0$. Then $\E[X]=0= \E[Y]$, $F_Y^{-1}(p)=0$, and
\begin{equation*}
   F_X^{-1}(p) = \left\{\begin{array}{ll} -a, & p \in [0, 0.5], \\ a, & p \in (0.5, 1], \end{array} \right.
\end{equation*}
By \eqref{eq-251220-1}, we obtain
\begin{equation*}
     F_X^{[-4]}(p) = \frac{1}{2} \int_0^p F_X^{-1}(u) (p -u)^2 \d u
     = \left\{\begin{array}{ll} -a p^3/6, & p \in [0, 0.5], \\ a (2 (p-0.5)^3 -p^3)/6, & p\in (0.5, 1], \end{array} \right.
\end{equation*}
and $F_Y^{[-4]}(p)=0$ for $p\in [0,1]$. It is easy to see that $F_X^{[-4]}(p) \le 0 = F_Y^{[-4]}(p)$ for all $p\in [0,1]$. Thus, $X\le^{-}_4 Y$. Through calculation, we have $\mu^X_{1:2}=-a/2 < 0=\mu^Y_{1:2}$ and $\mu^X_{1:3} =-3a/4 < \mu^Y_{1:3}$.  }
\end{example}

\begin{example}
{\rm Consider two random variables $X$ and $Y$ with $\P(X = 0)= 0.2$, $\P(X = 4) =0.5$, $\P(X = 5) = 0.3$ and $\P(Y = 1) = 0.2$, $\P(Y = 3) = 0.5$, $\P(Y = 6) = 0.3$. Then $\E[X]=3.5= \E[Y]$, and
\begin{equation*}
   F_X^{-1}(p) = \left\{\begin{array}{ll} 0, & p\in [0, 0.2], \\ 4, & p\in (0.2, 0.7], \\ 5, & p\in (0.7, 1], \end{array} \right. \qquad
   F_Y^{-1}(p) =\left\{\begin{array}{rl} 1, & p\in [0, 0.2], \\ 3, & p\in (0.2, 0.7], \\ 6, & p \in (0.7, 1]. \end{array} \right.
\end{equation*}
By \eqref{eq-251220-1}, we have
$$
 6 F_X^{[-4]}(p) = \left\{\begin{array}{ll}  0, & p\in [0,0.2],\\
          4(p-0.2)^3, & p\in (0.2, 0.7],\\
          (p-0.7)^3+ 4(p-0.2)^3, & p\in [0.7, 1],  \end{array}\right.
$$
and
\begin{equation*}
  6 F_Y^{[-4]}(p) = \left\{\begin{array}{ll}  p^3, & p\in [0, 0.2], \\
        2(p- 0.2)^3 + p^3, & p \in (0.2, 0.7], \\
        (p - 0.7)^3 + 2 (p- 0.2)^3 + p^3, & p\in (0.7, 1]. \end{array} \right.
\end{equation*}
It is straightforward to verify that $F_X^{[-4]}(p) \le F_Y^{[-4]}(p)$ for all $p \in [0,1]$, that is, $X\le^{-}_4 Y$. Through calculation, we have $\mu^X_{1:2}=2.65> 2.55=\mu^Y_{1:2}$ and $\mu^X_{1:3}=2.075<2.105=\mu^Y_{1:3}$.  }
\end{example}

\cite{DC10} introduced a new stochastic order called \emph{strong $n$-ISD} such that such an ordering can be characterized in terms of the $n$-majorization of the vectors of mean order statistics. $X$ is said to be dominated by $Y$ in the strong $n$-ISD, denoted by $X\le_n^{-\ast} Y$, if $X\le_n^{-} Y$ and $\mu^X_{1:j}=\mu^Y_{1:j}$ for $j=1, \ldots, n-1$.
Similarly,  $X$ is said to be strictly dominated by $Y$ in the strong $n$-ISD, denoted by $X<_n^{-\ast} Y$, if $X <_n^{-} Y$ and $\mu^X_{1:j}=\mu^Y_{1:j}$ for $j=1, \ldots, n-1$. The next is an example of $X <_3^{-\ast} Y$.

\begin{example}
{\rm ($X <_3^{-\ast} Y$)\ \ Consider two random variables $X$ and $Y$ with $\P(X = 0)= 0.2$, $\P(X = 4) =0.5$, $\P(X = 5) = 0.3$ and $\P(Y = 1) = 0.2$, $\P(Y = 13/4) = 0.5$, $\P(Y = 67/12) = 0.3$. Then $\E[X]=3.5= \E[Y]$, $\mu^X_{1:2}=2.65=\mu^Y_{1:2}$, and
\begin{equation*}
   F_X^{-1}(p) = \left\{\begin{array}{ll} 0, & p\in [0, 0.2], \\ 4, & p\in (0.2, 0.7], \\ 5, & p\in (0.7, 1], \end{array} \right. \qquad
   F_Y^{-1}(p) =\left\{\begin{array}{ll} 1, & p\in [0, 0.2], \\ 13/4, & p\in (0.2, 0.7], \\ 67/12, & p \in (0.7, 1]. \end{array} \right.
\end{equation*}
By \eqref{eq-251220-1}, we have
$$
  F_X^{[-3]}(p) = \left\{\begin{array}{ll}  0, & p\in [0,0.2],\\
          2(p-0.2)^2, & p\in (0.2, 0.7],\\
          \frac {1}{2}+2(p-0.7) + \frac {5}{2} (p-0.7)^2, & p\in [0.7, 1],  \end{array}\right.
$$
and
\begin{equation*}
   F_Y^{[-3]}(p) = \left\{\begin{array}{ll}  \frac {1}{2} p^2, & p\in [0, 0.2], \\[3pt]
        \frac {1}{50} + \frac {1}{5} (p- 0.2) + \frac {13}{8} (p-0.2)^2, & p \in (0.2, 0.7], \\[3pt]
        \frac {421}{800}  +\frac {73}{40} (p-0.7)+  \frac {67}{24} (p- 0.7)^2, & p\in (0.7, 1]. \end{array} \right.
\end{equation*}
For $p\in (0.2, 0.7]$, set $s=p-0.2\in (0, 0.5]$. Then
$$
   F_Y^{[-3]}(p)- F_X^{[-3]}(p)= \frac {1}{50} + \frac {1}{5} s - \frac {3}{8} s^2=: d_1(s)\ge \min\{d_1(0), d_1(0.5)\}>0.
$$
For $p\in (0.7, 1]$, set $t=p-0.7\in (0,0.3]$, Then
$$
   F_Y^{[-3]}(p)- F_X^{[-3]}(p)= \frac {21}{800} -\frac {7}{40} t+  \frac {7}{24} t^2=: d_2(t)\ge d_2(0.3)=0.
$$
Therefore, $F_X^{[-3]}(p) \le F_Y^{[-3]}(p)$ for all $p \in [0,1]$, that is, $X <^{-\ast}_3 Y$.   }
\end{example}

The next theorem states that if $X<_n^{-\ast} Y$ for $n>2$, then $\mu^X_{1:n} >\mu^Y_{1:n}$.

\begin{theorem}
 \label{th-251227}
Let $X, Y\in L^1$ satisfying $X <_n^- Y$ for $n>2$.
\begin{itemize}
   \item[{\rm (i)}] If $\mu^X_{1:k}=\mu^Y_{1:k}$ for $k=1, \ldots, n-1$, then $\mu^X_{1:n}> \mu^Y_{1:n}$.

   \item[{\rm (ii)}] If $\mu^X_{1:k} = \mu^Y_{1:k}$ for $k = 2, \ldots, n$, then $ (-1)^n \E\[X\] < (-1)^n \E\[Y\]$.
\end{itemize}
\end{theorem}

\begin{proof}
Since $X <_n^- Y$, it follows that $F^{[-n]}_X (p) \leq F^{[-n]}_Y (p)$ for all $p \in (0, 1)$, and there exists a point $p_0 \in (0,1)$ and $\alpha > 0$ such that $F^{[-n]}_Y (p_0) - F^{[-n]}_X (p_0) = 2 \alpha > 0$. Since $F^{[-n]}_X (p)$ is continuous in $p \in (0, 1)$, there exists an interval $[p_0, p_1] \subset (0, 1),\ p_0 < p_1$, such that $F^{[-n]}_Y (p) \ge F^{[-n]}_X (p) + \alpha$ for all $p \in [p_0, p_1]$. Thus, for $p > p_1$,
\begin{equation}
  \label{eq-251230-1}
   F^{[-n-1]}_Y (p) - F^{[-n-1]}_X (p) \geq \int_{p_0}^{p_1} \[F^{[-n]}_Y (u) -F^{[-n]}_X (u)\] \d u \geq \alpha (p_0 - p_1) > 0.
\end{equation}
Now we prove part (i) by contradiction. Assume on the contrary that $\mu^X_{1:n}\le \mu^Y_{1:n}$. Consider the following two cases.

First, assume $n$ is odd. By \eqref{eq-251230-2}, we have
\begin{align*}
  A_{n+1}^X(p) &= \frac {(1-p)^n}{n!} \sum^n_{j=1} (-1)^j {n\choose j} (1-p)^{-j} \mu^X_{1:j} \\
    &= -\frac {1}{n!} \mu^X_{1:n} + \frac {(1-p)^n}{n!} \sum^{n-1}_{j=1} (-1)^j {n\choose j} (1-p)^{-j} \mu^X_{1:j} \\
    & \ge -\frac {1}{n!} \mu^Y_{1:n} + \frac {(1-p)^n}{n!} \sum^{n-1}_{j=1} (-1)^j {n\choose j} (1-p)^{-j} \mu^Y_{1:j}
    = A_{n+1}^Y(p).
\end{align*}
Since $\widetilde{F}_Y^{[-n-1]}(p) \to 0$ as $p \to 1$ for $n>2$, we have, for any $0 < \ep < \alpha (p_0 - p_1)/2$,
$$
  F^{[-n-1]}_X (p) \ge A_{n+1}^X(p) -\ep \ge A_{n+1}^Y(p)-\ep > F^{[-n-1]}_Y (p) - 2\ep
$$
when $p \to 1$. This contradicts \eqref{eq-251230-1}, implying  $\mu^X_{1:n}> \mu^Y_{1:n}$ for odd $n>2$.

Next, assume $n$ is even. By \eqref{eq-251230-3}, we have
\begin{align*}
B_{n+1}^X(p) &= - \frac {(1-p)^n}{n!} \sum^n_{j=1} (-1)^j {n\choose j} (1-p)^{-j} \mu^X_{1:j} \\
   &= -\frac {1}{n!} \mu^X_{1:n} - \frac {(1-p)^n}{n!} \sum^{n-1}_{j=1} (-1)^j {n\choose j} (1-p)^{-j} \mu^X_{1:j} \\
   & \ge -\frac {1}{n!} \mu^Y_{1:n} - \frac {(1-p)^n}{n!} \sum^{n-1}_{j=1} (-1)^j {n\choose j} (1-p)^{-j} \mu^Y_{1:j}
    = B_{n+1}^Y(p).
\end{align*}
Similarly, for any $0 < \ep < \alpha (p_0 - p_1)/2$, we have $F^{[-n-1]}_X (p) > B_{n+1}^X(p) - \ep \ge B_{n+1}^Y(p) - \ep \ge F^{[-n-1]}_Y (p) - 2\ep$ when $p \to 1$.  This also contradicts \eqref{eq-251230-1}, implying  $\mu^X_{1:n}> \mu^Y_{1:n}$ for even $n>2$.

(ii)\ The proof is similar to that of Part (i), and hence omitted.
\end{proof}

In Theorem \ref{th-251227}, the conclusion is incorrect when $n=1,2$.

\section{The $n$-SD under independent noise}
\label{sect-4}

Denote by $\mathcal{P}_n$ the collection of all Borel probability measures on $\R$ that have finite $n$th moment, and by $\mathcal{M}_n$ the collection of all bounded Borel signed measures on $\R$ that have finite $n$th moment. Recall that a signed measure $\mu$ is bounded if its absolute value $|\mu|$ is a finite measure. For $\mu_1, \mu_2\in\mathcal{M}_n$, denote by $\mu_1\ast \mu_2$ the convolution of $\mu_1$ and $\mu_2$.

The following lemma will be used in the proof of the main result (Theorem \ref{th-251212}). It states that a signed measure on $\R$ with total mass one can be smoothed into a probability measure by convoluting another appropriately chosen probability measure.

\begin{lemma}
 \label{le-251228}
{\rm \citep{RS08, Mat04, PST20} }\ \ For every $\nu\in\mathcal{M}_n$ with $\nu(\R)=1$ and $n\in \{0, 1, \ldots, \oo\}$, there exists a $\mu\in \mathcal{P}_n$ such that $\mu\ast \nu\in \mathcal{P}_n$.
\end{lemma}

\begin{theorem}
 \label{th-251212}
Let $X, Y \in L^n$ for $n \ge 1$. If $\E[X^k] = \E[Y^k]$ for $k=0,\dots, n-1$ and
$$    (-1)^{n-1} \E[X^n] > (-1)^{n-1} \E[Y^n],     $$
then there exists a random variable $Z$, independent of $X$ and $Y$, such that $X + Z >_n Y + Z$.
\end{theorem}

\begin{proof}
Note that, for any $c \in \R$, $c^n = c_+^n + (-1)^n c_-^n$ and the integral representations
$$
  c_+^n =n\int_0^\oo (c-t)_+^{n-1}\d t=n\int_0^\oo\[(c-t)^{n-1} +(-1)^n(t-c)_+^{n-1}\] \d t
$$
and
$$
   c_-^n = n \int^0_{-\oo} (t - c)_+^{n-1} \d t.
$$
Using these representations, we obtain
\begin{align*}
  \frac{1}{n} \E[X^n] &= \int_0^\oo \E\[(X-t)^{n-1} +(-1)^n (t-X)_+^{n-1}\] \d t
       + (-1)^n \int_{-\oo}^0 \E\[(t-X)_+^{n-1}\] \d t   \\
     &= \int_0^\oo \left\{\E\[(X-t)^{n-1}\] + (-1)^n (n-1)! F_X^{[n]}(t)\right\} \d t
          + (-1)^n (n-1)!\int_{-\oo}^0 F_X^{[n]}(t) \d t
\end{align*}
Therefore,
\begin{align*}
   \frac{(-1)^n}{n!} \(\E[Y^n] - \E[X^n]\) &= \int_0^\oo \left\{\frac{(-1)^n}{(n-1)!} \E\[(Y-t)^{n-1}\]
       + F_Y^{[n]}(t) \right\}\d t + \int_{-\oo}^0 F_Y^{[n]}(t) \d t \\[3pt]
   & \quad -\int_0^\oo \left\{ \frac{(-1)^n}{(n-1)!} \E\[(X - t)^{n-1}\] + F_X^{[n]}(t)\right\} \d t
         - \int_{-\oo}^0 F_X^{[n]}(t) \d t\\[3pt]
   &=\frac{(-1)^n}{(n-1)!} \int_0^\oo\left \{\E\[(Y-t)^{n-1}\] -\E\[(X-t)^{n-1}\]\right\}\d t\\[3pt]
    &\quad + \int^\oo_{-\oo} \[F_Y^{[n]}(t) - F_X^{[n]}(t)\] \d t \\[3pt]
   &= \int_{-\oo}^\oo \[F_Y^{[n]}(t) - F_X^{[n]}(t)\] \d t,
\end{align*}
where the third equality follows from the assumption of $\E[X^k] = \E[Y^k]$ for $k=0,\ldots, n-1$.
Consequently,
\begin{equation*}
\int_\R \[ F_Y^{[n]}(t) - F_X^{[n]}(t)\] \d t =\frac{(-1)^n}{n!} \(\E[Y^n] - \E[X^n]\) =: \gamma,
\end{equation*}
which is positive by the assumption. Now define a signed measure $\nu$ by
\begin{equation*}
    \nu(A) :=\frac{1}{\gamma} \int_A \[F_Y^{[n]}(t)-F_X^{[n]}(t)\]\d t,\quad A\in \mathscr{B}(\R).
\end{equation*}
Hence, $\nu(\R) = 1$. Let $\mu_{{}_X}$ and $\mu_{{}_Y}$ be the Lebesgue-Stieltjes measures induced by distribution functions $F_X$ and $F_Y$, respectively. In what follows, $\d\, |F_X(x)-F_Y(x)|$ represents $\d\, |\mu_{{}_X} (x) -\mu_{{}_Y}(x)|$.
We conclude that $\nu$ is a bounded signed measure, since
\begin{align*}
 \gamma |\nu|(\R) &= \int_\R \left| F_Y^{[n]}(t) - F_X^{[n]}(t)\right| \d t  \\[3pt]
  &= \int_{-\oo}^\oo \left |\int_{-\oo}^t \frac{(t-u)^{n-1}}{(n-1)!}\d\(F_Y(u)-F_X(u)\) \right|\d t\\[3pt]
  &= \int_0^\oo \left|\int_{-\oo}^t \frac{(t-u)^{n-1}}{(n-1)!} \d \(F_Y(u)-F_X(u)\)\right| \d t
   +\int_{-\oo}^0\left|\int^t_{-\oo}\!\frac{(t-u)^{n-1}}{(n-1)!} \d\(F_Y(u)\!-\! F_X(u)\)\right|\d t\\[3pt]
   &\le \int_0^\oo\!\! \int^\oo_t \frac{(u-t)^{n-1}}{(n-1)!} \d\left| F_Y(u)-F_X(u) \right| \d t
      + \int_{-\oo}^0\!\int_{-\oo}^t \frac{(t-u)^{n-1}}{(n-1)!} \d\left|F_Y(u)-F_X(u)\right| \d t\\[3pt]
  & = \int_0^\oo\! \[\int^u_0 \frac{(u-t)^{n-1}}{(n-1)!} \d t\]\d\left| F_Y(u)-F_X(u) \right|
      + \!\int_{-\oo}^0 \[\int^0_u \frac{(t-u)^{n-1}}{(n-1)!}\d t\] \d\left|F_Y(u)-F_X(u)\right|\\[3pt]
   &= \int_{-\oo}^\oo \frac{|u|^n}{n!} \d \left| F_Y(u)-F_X(u) \right| \\
   &\le \frac{1}{n!} \(\E[|Y|^n] + \E[|X|^n]\) < \oo,
\end{align*}
where the third equality follows from the assumption of $\E[X^k] = \E[Y^k]$ for $k=0,\ldots, n-1$.
Therefore, $\nu \in {\cal M}_0$. Then, by Lemma \ref{le-251228}, there exists a probability measure $\mu \in {\cal P}_0$ such that $\nu\ast\mu \in {\cal P}_0$. Let $Z$ be a random variable independent from $X$ and $Y$ with distribution $\mu$. The signed measure $\nu$ is by definition absolutely continuous with density $\big [F_Y^{[n]}(t)-F_X^{[n]}(t)\big ]/\gamma$. Therefore, $\nu\ast\mu$ is absolutely continuous as well, and its density function $f(x)$ satisfies, almost everywhere,
\begin{align*}
  \gamma f(x) &=\int_\R \[F_Y^{[n]}(x-t) -F_X^{[n]}(x-t)\] \d \mu(t)   \\[3pt]
   &=\frac{1}{(n-1)!} \int_\R \Big\{\E\[(x-t-Y)_+^{n-1}\]-\E\[(x-t-X)_+^{n-1}\]\Big\}\d F_Z(t)\\[3pt]
   &=\frac{1}{(n-1)!} \Big \{\E\[(x-Z-Y)_+^{n-1}\] -\E\[(x-Z-X)_+^{n-1}\] \Big \}\\[3pt]
   &= F_{Y+Z}^{[n]}(x) - F_{X+Z}^{[n]}(x).
\end{align*}
Since $\nu\ast\mu$ is a probability measure, then $f(x)\ge 0$ and hence $F_{Y+Z}^{[n]}(x)\ge  F_{X+Z}^{[n]}(x)$ for almost every $x$. Since $F_{Y+Z}^{[n]}(x)$ and $F_{X+Z}^{[n]}(x)$ are right-continuous, this implies $F_{Y+Z}^{[n]}(x)\ge  F_{X+Z}^{[n]}(x)$ for all $x\in\R$. Furthermore, this inequality is strict somewhere since $f(x)$ is a density function. Therefore, $X+Z >_n Y+Z$.
\end{proof}

An immediate consequence of Theorem \ref{th-251212} for $n=1$ and $2$ is the following corollary.

\begin{corollary}
 \label{co-251201}
{\rm\citep{PST20}}. Let $X$ and $Y$ be two random variables.
\begin{itemize}
  \item[{\rm (i)}] If $X$ and $Y$ have finite expectations with $\E [X] >\E [Y]$, then there exists  a random variable $Z$, independent of $X$ and $Y$, such that $X+Z>_1 Y+Z$.
  \item[{\rm (ii)}] If $X$ and $Y$ have finite variances with $\E [X] =\E [Y]$ and $\Var(X)<\Var(Y)$, then there exists  a random variable $Z$, independent of $X$ and $Y$, such that $X+Z>_2 Y+Z$.
\end{itemize}
\end{corollary}

\cite{PST20} applied Corollary \ref{co-251201} to provide a simple axiomatization of the classic mean-variance preferences of \cite{Mark52}.

Motivated by Theorems \ref{th-251227}(i) and \ref{th-251212}, we have the following question.

\begin{question}
Let $X, Y \in L^1$ satisfy $\mu^X_{1:k}=\mu^Y_{1:k}$ for $k=1, \ldots, n-1$, and $\mu^X_{1:n}> \mu^Y_{1:n}$, where $n>2$. Is  there a random variable $Z$, independent of $X$ and $Y$, such that $X + Z <_n^{-} Y + Z$?
\end{question}

\section*{Funding}
	
Z. Zou, Corresponding author, gratefully acknowledges financial support from National Natural Science Foundation of China (No. 12401625), and the Fundamental Research Funds for the Central Universities (No. WK2040000108). T. Hu gratefully acknowledges financial support from the National Natural Science Foundation of China (No. 72332007, 12371476).

\section*{Disclosure statement}

No potential conflict of interest was reported by the authors.

\appendix

\setcounter{table}{0}
\setcounter{figure}{0}
\setcounter{equation}{0}
\renewcommand{\thetable}{A.\arabic{table}}
\renewcommand{\thefigure}{A.\arabic{figure}}
\renewcommand{\theequation}{A.\arabic{equation}}

\renewcommand{\thelemma}{A.\arabic{lemma}}
\renewcommand{\theexample}{A.\arabic{example}}
\renewcommand{\thecorollary}{A.\arabic{corollary}}
\renewcommand{\theremark}{A.\arabic{remark}}
\renewcommand{\thedefinition}{A.\arabic{definition}}

\begin{center}
\LARGE	{\bf Appendices}
\end{center}

\section{Asymptotes in the outcome-risk diagram}
 \label{sect-Asymptotes}

Define two functions as follows
\begin{equation}
  \label{eq-181228-3}
  C^X_n (x)= F_X^{[n]}(x) -\Fbar_X^{[n]}(x),\quad x\in\R,
\end{equation}
and
\begin{equation}
  \label{eq-181228-5}
  D^X_n (x)= F_X^{[n]}(x) + \Fbar_X^{[n]}(x),\quad x\in\R.
\end{equation}
If $n$ is even and $X\in L^{n-1}$, then applying \eqref{eq-181228-1} and \eqref{eq-181228-2} yields that
\begin{equation}   \label{eq-181228-4}
  C_n^X (x) = \frac {1}{(n-1)!} \Big \{ \E [(x-X)_+^{n-1}] -\E [(X-x)_+^{n-1}]\Big \}
      = - \frac {1}{(n-1)!} \E [(X-x)^{n-1}].
\end{equation}
Similarly, if $n$ is odd and $X\in L^{n-1}$, we have
\begin{equation}
  \label{eq-181228-6}
   D_n^X (x)=\frac {1}{(n-1)!} \E [(X-x)^{n-1}],\quad x\in\R.
\end{equation}
Note that $F^{[n]}_X(x)$ is increasing, continuous and convex function, while $\Fbar_X^{[n]}(x)$ is decreasing, continuous and convex function. Thus, as $x\to\oo$,
\begin{align*}
     F_X^{[n]}(x) -C_n^X (x) =\Fbar_X^{[n]}(x) & \longrightarrow 0\ \ \hbox{for even $n$}, \\
     D_n^X (x) -F^{[n]}(x) =\Fbar_X^{[n]}(x) & \longrightarrow 0\ \ \hbox{for odd $n$},
\end{align*}
which implies that  $C_n^X (x)$ defined by \eqref{eq-181228-3} is the \emph{right lower asymptote} of $F_X^{[n]}(x)$ for even $n$, and $D_n^X(x)$ defined by \eqref{eq-181228-5} is the \emph{right upper asymptote} of $F_X^{[n]}(x)$ for odd $n$.

We refer to the graph of $F^{[n]}_X(x)$ as the outcome-risk diagram. Denote by $\mu_X=\E X$ for $X\in L^1$ and $\sg_X^2=\Var(X)$ for $X\in L^2$. It is known from \cite{OR99} that for $n=2$, $C_2^X(x)=x-\mu_X$ is the right lower asymptote of $F_X^{[2]}(x)$ as $x$ goes to infinity. For $n=3$, \cite{GK00} pointed out that
$$   D_3^X(x)=\frac {1}{2} (x-\mu_X)^2 +\frac {1}{2}  \sigma_X^2  $$
is the right upper asymptote of $F_X^{[3]}(x)$ as $x$ goes to infinity. It should be pointed out that for $n\ge 1$, the $x$-axis is the left lower asymptote of $F_X^{[n]}(x)$ as $x$ goes to negative infinity.

\section{Proofs}
 \label{sect-proofs}

\subsection{Proof of Theorem \ref{th-181225}}

For convenience, we denote $\mu_j^X =\E [X^j]$ for $j\ge 1$, and set $\mu_i=\mu_i^X =\mu_i^Y$ for $i=1,\ldots, k$. We will prove the desired result \eqref{eq-181228-8} by contradiction. Assume on the contrary that \eqref{eq-181228-8} does not hold, that is,
\begin{equation}
  \label{eq-181228-9}
      (-1)^{k+1}\, \mu^X_{k+1} > (-1)^{k+1}\, \mu^Y_{k+1}.
\end{equation}
We consider two cases based on whether $n$ is odd or even.

First, assume $n$ is odd. In this case, let $D_n^X (x)$ and $D_n^Y (x)$ be the right upper asymptotes of $F_X^{[n]}(x)$ and $F_Y^{[n]}(x)$ as $x$ goes to infinity, respectively. From \eqref{eq-181228-6}, it follows that
\begin{align*}
  D_n^X (x) &= \frac {1}{(n-1)!} \left [\sum^k_{j=0} (-1)^j {n-1\choose j} \mu_j x^{n-1-j} + \sum^{n-1}_{j=k+1} (-1)^j {n-1\choose j} \mu_j^X x^{n-1-j}\right ],\\
  D_n^Y (x) &= \frac {1}{(n-1)!} \left [\sum^k_{j=0} (-1)^j {n-1\choose j} \mu_j x^{n-1-j} + \sum^{n-1}_{j=k+1} (-1)^j {n-1\choose j} \mu_j^Y x^{n-1-j}\right ].
\end{align*}
Then, as $x\to\oo$,
\begin{align*}
  D_n^X (x) - D_n^Y (x) & = \frac {1}{(n-1)!} \sum^{n-1}_{j=k+1} (-1)^j {n-1\choose j} (\mu_j^X -\mu_j^Y) x^{n-1-j}\\
      &= \frac {1}{(n-1)!} (-1)^{k+1} {n-1\choose k+1} \(\mu_{k+1}^X -\mu_{k+1}^Y\) x^{n-k-2} \left [1+ \circ (1) \right ]
      \longrightarrow \eta_k>0,
\end{align*}
where
\begin{eqnarray*}
   \eta_k = \left \{\begin{array}{ll} \dfrac {1}{(n-1)!} {n-1\choose k+1}  (-1)^{k+1} \(\mu_{k+1}^X -\mu_{k+1}^Y\) > 0, & k=n-2,\\
            +\oo, & k<n-2.   \end{array}   \right.
\end{eqnarray*}
Thus, there exist $\epsilon_0>0$ and $x_0\in\R$ such that $D_n^X (x)- D_n^Y (x)>\epsilon_0$ and $ D_n^X (x)- F_X^{[n]}(x) < \epsilon_0$ for all $x>x_0$. So, we have
$$
   F_X^{[n]}(x)> D_n^X (x)-\epsilon_0 > D_n^Y (x)\ge F_Y^{[n]}(x),\quad x>x_0,
$$
which contradicts with $X\ge_n Y$. Therefore, \eqref{eq-181228-8} holds when $n$ is odd.

Next, assume $n$ is even. We consider the right lower asymptotes $D_n^X (x)$ and $D_n^Y (x)$ of $F^{[n]}_X(x)$ and $F_Y^{[n]}(x)$, respectively. Applying \eqref{eq-181228-4} yields that, as $x\to\oo$,
\begin{align*}
  C_n^X (x) - C_n^Y (x) &= \frac {1}{(n-1)!} \sum^{n-1}_{j=k+1} (-1)^j {n-1\choose j} \(\mu_j^X -\mu_j^Y\) x^{n-1-j}\\
  &= \frac {1}{(n-1)!} (-1)^{k+1} {n-1\choose k+1} \(\mu_{k+1}^X -\mu_{k+1}^Y\) x^{n-k-2} \left [1+\circ(1) \right ]
  \longrightarrow \eta_k>0.
\end{align*}
Thus, there exist $\epsilon_1>0$ and $x_1\in\R$ such that $C_n^X (x)- C_n^Y (x) >\epsilon_1$ and $F_Y^{[n]}(x) - C_n^Y (x)< \epsilon_1$ for all $x>x_1$. So, we have
$$  F_X^{[n]}(x)\ge C_n^X (x) > C_n^Y (x)+\epsilon_1> F_Y^{[n]}(x),\quad x>x_1,   $$
contradicting $X\ge_n Y$. Therefore, \eqref{eq-181228-8} holds when $n$ is odd. This completes the proof of
Theorem \ref{th-181225}. \ $\square$

\subsection{Proof of Theorem \ref{th-181235}}

We use a similar idea to that in the proof of Theorem \ref{th-251227}). Note that $X>_n Y$ implies that $F_X^{[n]}(x)\le F_Y^{[n]}(x)$ for all $x\in\R$, and there exists a point $x_0\in\R$ such that $F_Y^{[n]}(x_0)-F_X^{[n]}(x_0)=2\alpha>0$. Since $F_X^{[n]}(x)$ is continuous in $x$, there exists an interval $[x_0, x_1]$, $x_0<x_1$, such that $F_Y^{[n]}(x) \ge F_X^{[n]}(x) +\alpha$ for all $x\in [x_0, x_1]$. Thus, for $x>x_1$,
\begin{equation}
 \label{eq-181229-5}
   F_Y^{[n+1]}(x)-F_X^{[n+1]}(x)  \ge  \int^{x_1}_{x_0} \big [F_Y^{[n]}(t)-F_X^{[n]}(t)\big ]\,\d t\ge  \alpha (x_1-x_0)>0.
\end{equation}
Now we prove \eqref{eq-181229-4} by the way of contradiction. Assume on the contrary that \eqref{eq-181229-4} does not hold, that is,
$
   (-1)^{n-1} \mu^X_n \le (-1)^{n-1}\mu^Y_n.
$
Consider the following two cases.

First, assume $n$ is even. Note that $D^X_{n+1}(x)$ and $D_{n+1}^Y(x)$ are the right upper asymptotes of $F^{[n+1]}(x)$ and $G^{[n+1]}(x)$, respectively. Also,
\begin{align*}
    D_{n+1}^X (x) &= \frac {1}{n!} \E [(X-x)^n] =\frac {1}{n!} \bigg [\mu_n^X +\sum^{n-1}_{j=0} {n\choose j} \mu^X_j (-1)^{n-j} x^{x-j}\bigg ]   \\[3pt]
    &\ge \frac {1}{n!} \bigg [\mu_n^Y +\sum^{n-1}_{j=0} {n\choose j}\mu_j^Y (-1)^{n-j} x^{x-j}\bigg ]
    =\frac {1}{n!}\, \E [(Y-x)^n] = D_{n+1}^Y(x).
\end{align*}
For any $0<\epsilon<\alpha (x_1-x_0)$, we have $F_X^{[n+1]}(x)> D_{n+1}^X(x)-\epsilon \ge D_{n+1}^Y (x)-\epsilon \ge F_Y^{[n+1]}(x)-\epsilon$ when $x$ is large enough. This contradicts \eqref{eq-181229-5}. We thus prove \eqref{eq-181229-4} when $n$ is even.

Next, assume $n$ is odd. Note that $C_{n+1}^X(x)$ and $C_{n+1}^Y(x)$ are right lower asymptotes of $F_X^{[n+1]}(x)$ and $F_Y^{[n+1]}(x)$, respectively. Similarly, we have
\begin{align*}
    C_{n+1}^X (x) &= \frac {1}{n!} \bigg [-\mu_n^X -\sum^{n-1}_{j=0} {n\choose j} \mu_{Y,j} (-1)^{n-j} x^{n-j}\bigg ]   \\[3pt]
    &\ge \frac {1}{n!} \bigg [-\mu_n^Y -\sum^{n-1}_{j=0} {n\choose j} \mu_j^Y (-1)^{n-j} x^{n-j}\bigg ]
    = - \frac {1}{n!} \E [(Y-x)^n] = C_{n+1}^Y (x).
\end{align*}
For any $0<\epsilon<\alpha (x_1-x_0)$, we have $F_X^{[n+1]}(x)\ge C_{n+1}^X(x)\ge C_{n+1}^Y(x) \ge F_Y^{[n+1]}(x)-\epsilon
$ when $x$ is large enough. This also contradicts \eqref{eq-181229-5}. We thus prove \eqref{eq-181229-4} when $n$ is odd. \ \qed

\end{document}